\documentclass[runningheads]{llncs}
\usepackage{url}
\usepackage{amssymb,latexsym,makeidx,amsmath,amsfonts}

\def\clP{\mathcal{P}}
\def\restr{{\upharpoonright}}
\newcommand\hull[1]{\overline{#1}}

\newcommand\hide[1]{}

\newcommand\comm[1]\empty
\newcommand\commOld[1]\empty

\def\iff{\Leftrightarrow}

\def\emp{\varnothing}

\def\frF{\framets{F}}
\def\Alg{\framets{A}}
\def\frA{\framets{A}}
\def\frG{\framets{G}}
\def\vL{L}
\newcommand\framets[1]{\mathsf{#1}}

\newcommand\classts[1]{\mathcal{#1}}

\def\EE{\exists}
\def\AA{\forall}

\def\Log{Log}
\def\ILog{{ILog}}

\def\clA{\classts{A}}
\def\clB{\classts{B}}
\def\clC{\classts{C}}
\def\clD{\classts{D}}

\author{Ilya  Shapirovsky}

\institute{Institute for Information Transmission Problems of the Russian Academy of Sciences, Moscow, Russia
\\ \smallskip \email{shapir@iitp.ru}}

\title{Modal logics of finite direct powers of $\omega$ have the finite model property}

\begin{document}

\maketitle
\begin{abstract}
Let $(\omega^n,\preceq)$ be the direct power of $n$ instances of  $(\omega,\leq)$, natural numbers with the standard ordering,
$(\omega^n,\prec)$  the direct power of $n$ instances of  $(\omega,<)$.
We show that for all finite $n$, the modal logics  of $(\omega^n,\preceq)$
and of $(\omega^n,\prec)$
have the finite model property and moreover, the modal algebras of the frames $(\omega^n,\preceq)$ and
$(\omega^n,\prec)$
are locally finite.


\end{abstract}




\hide{Is well-foundness implied by ``refinable''? Denis: yes.}

\section{Introduction}
We consider modal logics of direct products of linear orders.
It is known that the logics of finite direct powers of real numbers and of rational numbers with the standard non-strict ordering
have the finite model property, are finitely axiomatizable, and consequently are decidable.
These non-trivial results were obtained in
 \cite{GoldMink1980}, and independently in \cite{ShehtMink83}. \comm{More details: $n=2$}
Later, analogous results  were obtained for the logics of finite direct powers of $(\mathbb{R},<)$
  \cite{ShShChron03}.
Recently, it was shown that the direct squares $(\mathbb{R},\leq,\geq)^2$ and $(\mathbb{R},<,>)^2$ have decidable bimodal logics \cite{MinkTemp-Nonstr}, \cite{MinkTemp-Str}.
Direct products  of well-founded orders 
have never been considered before in the context of modal logic.

\smallskip

Let $(\omega^n,\preceq)$ be the direct power of $n$ instances of  $(\omega,\leq)$, natural numbers with the standard ordering: for $x,y\in \omega^n$, $x\preceq y$ iff $x(i)\leq y(i)$ for all $i<n$.
Likewise, let $(\omega^n,\prec)$ be the the direct power  $(\omega,<)^n$. 
 We will show that for all finite $n>0$, the logics  $\Log(\omega^n,\preceq)$
and
$\Log(\omega^n,\prec)$
have the finite model property and moreover, the algebras of the frames $(\omega^n,\preceq)$
and $(\omega^n,\prec)$
are
locally finite.

\section{Partitions of frames, local finiteness, and the finite model property}

We assume the reader is familiar with the basic notions of modal logics \cite{BRV-ML,CZ}.
By a {\em logic} we mean a normal propositional modal logic.
For a (Kripke) frame $\frF$, $\Log(\frF)$ denotes its modal logic, i.e., the set of all valid in $\frF$ modal formulas.
For a set $W$, $\clP(W)$ denotes the powerset of $W$. The {\em (complex) algebra of a frame} $(W,R)$ is the modal algebra
$(\clP(W),R^{-1})$. The algebra of $\frF$ is denoted by $\Alg(\frF)$.
A logic has the {\em finite model property} if it is complete with respect to a class of finite frame (equivalently,  finite algebras).

\smallskip

{\em A partition} $\clA$ of a set $W$ is a set
of non-empty pairwise disjoint sets such that $W=\bigcup \clA$. A partition $\clB$ {\em refines} $\clA$, if each element of $\clA$ is the 
 union of some elements of $\clB$.

\begin{definition}\label{def:tune}
\normalfont
Let $\frF=(W,R)$ be a Kripke frame.
A partition $\clA$ of $W$ is {\em tuned (in $\frF$)} if for every $U,V\in \clA$,
$$\EE u\in U \, \EE v\in V \, uR v  \;\Rightarrow\; \AA u\in U\,\EE v\in V\,uRv.$$
$\frF$ is {\em tunable} if for
 every finite partition $\clA$ of $\frF$ there exists a finite tuned refinement $\clB$ of $\clA$.

\end{definition}

\begin{proposition}\label{prop:Filtr}
If $\frF$ is tunable, then $\Log(\frF)$ has the finite model property.
\end{proposition}
Apparently, this fact was first observed by H. Franz\'{e}n (see \cite{Franz-Bull}). \comm{Fine? Double-check!}
For the proof, we notice the following:
a finite partition $\clB$ is tuned in a frame $(W,R)$ iff the family
$\{\cup x\mid x\subseteq \clB\}$ of subsets of $W$ forms a subalgebra of the modal algebra $(\clP(W), R^{-1})$.
Recall that an algebra $\frA$ is {\em locally finite} if every finitely generated subalgebra of $\frA$ is finite.
Thus, from Definition \ref{def:tune} we have:
\begin{proposition}\label{prop:tunableLF}
The algebra of a frame $\frF$ is locally finite iff $\frF$ is tunable.
\end{proposition}
If $\vL$ is the logic of a frame $\frF$,
then $\vL$ is the logic of the modal algebra $\Alg(\frF)$.  Equivalently, $\vL$ is the logic of finitely generated
subalgebras of $\Alg(\frF)$.   It follows that if $\Alg(\frF)$
is locally finite, then $\vL$ has the finite model property.

\smallskip
Thus, logics of tunable frames  have the finite model property,
and moreover, algebras of tunable frames are locally finite.
\comm{see Nick's terminology}

\begin{example}\label{ex:linear}
Consider the frame $(\omega,\leq)$, natural numbers with the standard ordering.
Suppose that $\clA$ is a finite partition of $\omega$. If every $A\in\clA$ is infinite, then $\clA$ is tuned in $(\omega,\leq)$ and in $(\omega,<)$.
Otherwise, let $k_0$ be the greatest element of the finite set $\bigcup \{A\in\clA\mid A \textrm{ is finite}\}$, and $U=\{k\mid k_0<k<\omega\}$.
Consider the following finite partition $\clB$ of $\omega$:
$$\clB=\{\{k\}\mid k\leq k_0\} \cup \{A\cap U\mid A \text{ is an infinite element of } \clA \}.$$
Each
element of $\clB$ is either infinite, or a singleton, and singletons in $\clB$ cover an initial segment of $\omega$. Thus, $\clB$ is a finite refinement of $\clA$ which is tuned in $(\omega,\leq)$ and in $(\omega,<)$.

It follows that the algebras of the frames $(\omega,\leq)$ and $(\omega,<)$ are locally finite.
\end{example}

\begin{remark}
Although the algebras
of the frames $(\omega,\leq)$ and $(\omega,<)$ are locally finite, the logics of these frames are not.

Recall that a logic $\vL$ is {\em locally finite} (or {\em locally tabular}) if the
Lindenbaum algebra of $\vL$ is locally finite.
A logic of a transitive frame is locally finite iff the frame is of finite height \cite{Seg_Essay},\cite{Max1975}.
It terms of tuned partitions, local finiteness of a logic is characterized  as follows \cite{LocalTab16AiML}:
the logic of a frame $\frF$ is locally finite
iff there exists a function $f:\omega\to \omega$ such that
 for every finite partition $\clA$ of $W$ there exists a tuned in $\frF$ refinement
$\clB$ of $\clA$ such that $|\clB|\leq f(|\clA|)$.

\end{remark}

\section{Main result}

\begin{theorem}\label{thm:omegaNpLocFin}
For all finite $n>0$, the algebras $\Alg(\omega^n,\preceq)$ and $\Alg(\omega^n,\prec)$ are locally finite.
\end{theorem}

The simple case $n=1$ was considered in Example \ref{ex:linear}.
To prove the theorem for the case of arbitrary finite $n$, we need some auxiliary constructions.

\begin{definition}\label{def:monot}\normalfont
\commOld{Orthogonal hull; orthant;Diophantine plane}
Consider a non-empty $V\subseteq \omega^n$.  \comm{dim-codim?}
Put
$$
\begin{array}{llll}
&J(V)&= \{i<n\mid \EE x\in V\,\EE y\in V\; x(i)\neq y(i)\}, \\
&I(V)&=\{i<n\mid \AA x\in V\,\AA y\in V\;  x(i)= y(i)\}=n{\setminus}J(V).
\end{array}
$$
The {\em hull of} $V$ is the set $$\hull{V}=\{y\in \omega^n \mid \AA i\in I(V)\, (y(i)=x(i) \textrm{ for some (for all) } x\in V)\}.$$
$V$ is {\em pre-cofinal} if it is cofinal in its hull, i.e.,
$$\AA x\in \hull{V} \, \EE y\in V \, x\preceq y.$$



A partition $\clA$ of $V\subseteq \omega^n$ is {\em monotone} if
\begin{itemize}
\item[-]  all of its  elements are pre-cofinal, and \comm{ENGL: all of its}
\item[-]  for all $x,y\in V$ such that $x\preceq y$ we have $J([x]_\clA)\subseteq J([y]_\clA)$,
\end{itemize}
where $[x]_\clA$ is the element of $\clA$ containing $x$.
\end{definition}

\hide{Correct version for preceq:
\begin{lemma} If $\clA$ is a monotone partition of  $\omega^n$, then $\clA$ is tuned.
\end{lemma}
\begin{proof}
Let $A,B\in \clA$, $x,y\in A$, $x\preceq z\in B$. Let $u$ be the following point in $\omega^n$:
\begin{eqnarray}\label{eqn:u-def}
u(i)=y(i)   \text{ for }i\in J(A), \text{ and }
u(i)=z(i)    \text{ for } i\in I(A).
\end{eqnarray}
We have $$\{i<n\mid u(i)\neq z(i)\} \subseteq J(A) \subseteq J(B);$$
the first inclusion follows from (\ref{eqn:u-def}), the second follows from the monotonicity of $\clA$.
Hence, we have $u(i)=z(i)$ for all $i\in I(B)$. By the definition of $\hull{B}$, we have $u\in \hull{B}$.
Since  $B$ is cofinal in $\hull{B}$  (we use monotonicity again), for some $u'\in B$ we have $u\preceq u'$.

By (\ref{eqn:u-def}), we have $y(i)\leq u(i)$ for all $i<n$: indeed, $y(i)=x(i)\leq z(i)$ for $i\in I(A)$, and  $u(i)= z(i)$ otherwise. Thus, $y\preceq u$, and so
 $y\preceq u'$, as required.
\qed
\end{proof}
}

\begin{lemma} If $\clA$ is a monotone partition of  $\omega^n$, then $\clA$ is  tuned in
$(\omega^n, \preceq)$ and in $(\omega^n, \prec)$.
\end{lemma}
\begin{proof}
Let $A,B\in \clA$, $x,y\in A$, $x\preceq z\in B$. Let $u$ be the following point in $\omega^n$:
\begin{eqnarray}\label{eqn:u-def}
u(i)=y(i)+1   \text{ for }i\in J(A), \text{ and }
u(i)=z(i)    \text{ for } i\in I(A).
\end{eqnarray}
We have \comm{Well, it seems to be unnecessary complicated}
$$\{i<n\mid u(i)\neq z(i)\} \; \subseteq \; n{\setminus}I(A)\; =\; J(A) \; \subseteq \; J(B);$$
the first inclusion follows from (\ref{eqn:u-def}), the second follows from the monotonicity of $\clA$.
Hence, we have $u(i)=z(i)$ for all $i\in I(B)$. By the definition of $\hull{B}$, we have $u\in \hull{B}$.
Since  $B$ is cofinal in $\hull{B}$  (we use monotonicity again), for some $u'\in B$ we have $u\preceq u'$.

\smallskip

By (\ref{eqn:u-def}), we have $y(i)\leq u(i)$ for all $i<n$: indeed, $y(i)=x(i)\leq z(i)=u(i)$ for $i\in I(A)$, and  $u(i)= y(i)+1$ otherwise. Thus, $y\preceq u$, and so
 $y\preceq u'$. It follows that $\clA$ is  tuned in $(\omega^n, \preceq)$.

\smallskip
In order to show that $\clA$ is  tuned in $(\omega^n, \prec)$, we now assume that $x\prec z$. Then we have
  $y(i)< u(i)$ for all $i<n$, since $y(i)=x(i)< z(i)=u(i)$ for $i\in I(A)$, and  $u(i)= y(i)+1$ otherwise.
  Hence $y\prec  u$. Since $u\preceq u'$, we have $y\prec u'$, as required.
  \comm{DOUBLE CHECK THE LEMMA!!!}
\qed
\end{proof}


Let $\clA$ be a partition of a set $W$.
For  $V\subseteq W$, the partition $$\clA\restr V=\{A\cap V\mid  A\in \clA\;\&\;A\cap V\neq \emp\}$$ of $V$ is called the {\em restriction of $\clA$ to $V$}.
For a family $\clB$ of subsets of $W$, the {\em partition induced by $\clB$ on $V\subseteq W$} is the  quotient set
$V/{\sim}$, where $$x\sim y \text{ iff }  \AA A\in\clB \, (x\in A\iff y\in A).$$

\begin{lemma}\label{lem:monpartexists}
If  $\clA$ is a finite partition of $\omega^n$, then there exists its finite monotone refinement.
\end{lemma}
\begin{proof}
By induction on $n$.

\smallskip
Suppose $n=1$.
Let $k_0$ be the greatest element of the finite set $$\bigcup \{A\in\clA\mid A \textrm{ is finite}\}.$$
Put $\clB=\{\{k\}\mid k\leq k_0\}\cup\{k\mid k_0<k < \omega\}$.
Let $\clC$  be the
the partition induced by $\clA\cup\clB$ on $\omega$.
Consider $x\in \omega$ and put $A=[x]_\clC$. If $x\leq k_0$, then $A=\hull{A}=\{x\}$ and  $J(A)=\emp$.
If  $x>k_0$, then $A$ is cofinal in $\omega$, $\hull{A}=\omega$, $J(A)=\{0\}$.
In follows that $\clC$  is the required monotone refinement of $\clA$.

\smallskip
Suppose $n>1$.
For $k\in \omega$ let $U_{k}=\{y\in\omega^n\mid y(i)\geq k \text{ for all }i<n\}$.
Since $\clA$ is finite, we can chose a natural number $k_0$ such that
\begin{center}
if $y\in U_{k_0}$, then $[y]_\clA$ is cofinal in $\omega^n$.
\end{center}
It follows that the partition $\clA\restr U_{k_0}$  is monotone: it consists of cofinal in $\omega^n$ sets which are obviously pre-cofinal, and $J(A)=n$ for all $A\in\clA\restr U_{k_0}$.

We are going to extend this partition step by step
in order to obtain
a sequence of finite monotone partitions of $U_{k_0-1}$, \ldots, $U_{0}=\omega^n$, respectively refining
$\clA\restr U_{k_0-1}, \ldots, \clA\restr U_{0}=\clA$.\comm{Engl}


First, let us describe the construction for the case $k_0=1$, the crucial technical step of the proof.

\smallskip
\noindent{\it Claim A.}
Suppose that $\clB$ is a finite monotone partition of $U_1$ refining $\clA\restr U_1$. Then there exists a finite monotone
partition $\clC$ of $\omega^n$ refining $\clA$ such that $\clB\subseteq \clC$.
\begin{proof}
$\clC$ will be the union of $\clB$ and a partition of the set
$$V=\{x\in\omega^n\mid x(i)=0 \text{ for some } i<n\} = \omega^n{\setminus}U_1.$$
To construct the required partition of $V$, for $I\subseteq n$ put $$V_I=\{x\mid \AA i<n\,(i\in I \,\iff\, x(i)=0)\}.$$
Then $\{V_I\mid \emp \neq I\subseteq n\}$ is a partition of $V$, $V_\emp=U_1$.

Each $V_I$ considered with the order $\preceq$ on it  is isomorphic to $(\omega^{n-|I|},\preceq)$.
Thus, by the induction hypothesis, for a non-empty $I\subseteq n$ we have:
\begin{equation}\label{eq:partVI}
\text{Each finite partition of $V_I$ admits a finite monotone refinement.\comm{details}}
\end{equation}

For $I\subseteq n$, by induction on the cardinality of $I$ we define a finite partition $\clC_I$ of $V_I$.

We put $\clC_\emp=\clB$.

Assume that $I$ is non-empty.
Consider the projection $\Pr_I: x\mapsto y$  such that
$y(i)=0$ whenever $i\in I$, and $y(i)=x(i)$ otherwise. Note that for all $K\subset I$,
$x\in V_K$ implies $\Pr_I(y)\in V_I$.
Let $\clD$ be the partition induced on $V_I$ by the family
\begin{equation}\label{eq:defOfClI}
\clA \cup \bigcup\limits_{K\subset I}\{\Pr\nolimits_I(A)\mid A\in V_K \}.
\end{equation}
By an immediate induction argument, $\clD$ is finite.
Let $\clC_I$ be a finite monotone refinement of $\clD$, which exists according to
(\ref{eq:partVI}).

We put
$$\clC=\bigcup_{I\subseteq n} \clC_I.$$
Then $\clC$ is a finite refinement of $\clA$. We have to check monotonicity.

Every element $A$ of $\clC$ is pre-cofinal, because $A$ is an element of a monotone partition $\clC_I$ for some $I$.
In order to check the second condition of monotonicity,
we consider $x,y$ in $\omega^n$ with $x\preceq y$ and  show  that
\begin{equation}\label{eq:monReq}
J([x]_\clC)\subseteq J([y]_\clC).
\end{equation}
Let $x \in V_I$, $y \in V_K$ for some $I,K\subseteq n$.
Since $x\preceq y$, we have $K\subseteq I$. If $K=I$, then (\ref{eq:monReq}) holds, since
$[x]_\clC$ and $[y]_\clC$ belong to the same monotone partition $\clC_I$.
Assume that $K\subset I$.  In this case we have:
$$J([x]_\clC)\subseteq J([\Pr\nolimits_I(y)]_\clC)\subseteq J(\Pr\nolimits_I([y]_\clC))\subseteq J([y]_\clC).$$
To check the first inclusion, we observe that $\Pr_I(y)$ belongs to $V_I$ (since $K\subset I$). This means that
$[x]_\clC$ and $[\Pr_I(y)]_\clC$ are elements of the same partition $\clC_I$.
We have $x\preceq \Pr_I(y)$, since $x\in V_I$ and $x\preceq y$. Now the first inclusion follows from monotonicity of $\clC_I$.
By (\ref{eq:defOfClI}),
 $\Pr\nolimits_I([y]_\clC)$  is the union of some elements of $\clC_I$ (since $K\subset I$ and $[y]_\clC\in \clC_K$); trivially, $\Pr_I(y)\in \Pr_I([y]_\clC)$,
 hence $[\Pr\nolimits_I(y)]_\clC$ is a subset of $\Pr\nolimits_I([y]_\clC)$. This yields the second inclusion.
The third inclusion is immediate from Definition \ref{def:monot}.
Thus, we have (\ref{eq:monReq}), which proves the claim.
\qed
\end{proof}

From Claim A one can easily obtain the following: \comm{details}

\smallskip
\noindent{\it Claim B.} Let $0<k<\omega$. If $\clB$ is a finite monotone partition of $U_k$ refining $\clA\restr U_k$, then there exists a finite monotone
partition $\clC$ of $U_{k-1}$ refining $\clA$ such that $\clB\subseteq \clC$.

\smallskip
Applying Claim B $k_0$ times, we obtain the required monotone refinement of $\clA$. This proves Lemma \ref{lem:monpartexists}.\qed
\end{proof}

From the above two lemmas we obtain that the frames $(\omega^n, \preceq)$ and $(\omega^n, \prec)$, $0<n<\omega$, are tunable.
Now the proof of the Theorem \ref{thm:omegaNpLocFin} immediately follows from Proposition \ref{prop:tunableLF}.
\comm{
\comm{Intervals staff...}
\comm{See {\em their} papers!}
}

\begin{corollary}\label{cor:omegaNfmp}
For all finite $n$, the logics $\Log(\omega^n,\preceq)$ and $\Log(\omega^n,\prec)$
have the finite model property.
\end{corollary}

\section{Open problems and conjectures}
It is well-known  that every extension of $\Log(\omega,\leq)$ has the finite model property \cite{Bull}.
\begin{question}\label{conj:extFMP}
Let $\vL$ be an extension of  $\Log(\omega^n,\preceq)$ for some finite $n>0$.
Does $\vL$ have the finite model property?
\end{question}

Every extension of a locally finite logic is locally finite, and so has the finite model property.
Although the algebras
of the frames $(\omega^n,\preceq)$ and $(\omega^n,\prec)$ are locally finite, the logics of these frames are not
(recall that a logic of a transitive frame is locally finite iff the frame is of finite height \cite{Seg_Essay},\cite{Max1975}).
Thus, Theorem \ref{thm:omegaNpLocFin} does not answer Question \ref{conj:extFMP}.

At the same time,  Theorem \ref{thm:omegaNpLocFin} yields another corollary.
A {\em subframe} of a frame $(W,R)$ is the restriction $(V,R\cap (V\times V))$, where $V$ is a non-empty subset of $W$.
It follows  from Definition \ref{def:tune} that if a frame is tunable then all its subframes are
(details can be found in the proof of Lemma 5.9 in \cite{LocalTab16AiML}). From Proposition \ref{prop:tunableLF}, we have:
\begin{proposition}
If the algebra of a frame $\frF$ is locally finite, then the algebras of all  subframes of $\frF$ are.
\end{proposition}

\begin{corollary}
For all finite $n$, if $\frF$ is a subframe of $(\omega^n,\preceq)$ or of  $(\omega^n,\prec)$, then
$\Alg(\frF)$ is locally finite, and $\Log(\frF)$ has the finite model property.
\end{corollary}

While $\Log(\omega,\leq)$  is not locally finite, the intermediate logic $\ILog(\omega,\leq)$ is.\comm{Cite?}
\begin{conjecture}
For all finite $n$,   $\ILog(\omega^n,\preceq)$ is locally finite.
\end{conjecture}

\medskip

The logics of finite direct powers of $(\mathbb{R},\leq)$ and of $(\mathbb{R},<)$ have the finite model property, are finitely axiomatizable, and consequently are decidable \cite{GoldMink1980}, \cite{ShehtMink83}, \cite{ShShChron03}. Recently, it was shown that the direct squares $(\mathbb{R},\leq,\geq)^2$ and $(\mathbb{R},<,>)^2$ have decidable bimodal logics \cite{MinkTemp-Nonstr}, \cite{MinkTemp-Str}.

\begin{question} Let $n>1$.
Is $\Log(\omega^n,\preceq)$ decidable? recursively axiomatizable?
Does $\Log(\omega^n,\preceq,\succeq)$ have the finite model property?
\end{question}

\begin{proposition}\label{prop:finiteXpart}
If $\frF$ is tunable and $\frG$ is finite, then $\frF\times \frG$ is tunable.
\end{proposition}
\begin{proof}\comm{In haste. Doublecheck.}
Let $\frF=(F,R)$, $\frG=(G,S)$, and
$\clA$ be a finite partition of $F\times G$.
For $A$ in $\clA$ and $y$ in $G$, we put
$\Pr_y(A)=\{x\in F\mid (x,y)\in A\}$, $\clA_y={\{ \Pr_y(A)\mid A\in \clA\}}$.
Let  $\clB$  be the partition induced on $F$ by the family
$\bigcup_{y\in G} \clA_y$. Since $\clB$ is finite, there exists
its finite tuned refinement $\clC$. Consider the partition $$\clD=\{A\times \{w\}\mid A\in \clC  \,\&\,  y\in G \}$$ of $F\times G$.
Then $\clD$ is a finite refinement of $\clA$.
It is not difficult to check that $\clD$ is tuned in $\frF\times \frG$.
\qed\end{proof}

If follows that if the algebra of $\frF$ is locally finite and $\frG$ is finite, then
the algebra of $\frF\times \frG$  is locally finite.
\begin{question}
Consider tunable frames $\frF_1$ and $\frF_2$.
Is the direct product $\frF_1\times\frF_2$ tunable?
\end{question}

If this is true, then Theorem \ref{thm:omegaNpLocFin} immediately follows from the simple one-dimensional case. And, moreover,
in this case Theorem \ref{thm:omegaNpLocFin} can be generalized to arbitrary ordinals in view of the following observation.
\begin{proposition}\label{prop:ord-part}
For every ordinal $\alpha>0$,
the modal algebras $\Alg(\alpha,\leq)$, $\Alg(\alpha,<)$  are locally finite.
\end{proposition}
\begin{proof}
By induction on $\alpha$ we show that the frames $(\alpha,\leq)$, $(\alpha,<)$  are tunable.

For a finite $\alpha$, the statement is trivial.

Suppose that $\clA$ is a finite partition of an infinite $\alpha$.
If every element of $\clA$ is cofinal in $\alpha$, then $\clA$ is tuned in $(\alpha,\leq)$ and in $(\alpha,<)$.
Otherwise,  we put
$$
\beta=\sup\bigcup\{A\in \clA\mid A \text{ is not cofinal in } \alpha \}.
$$
Since $\clA$  is finite, we have $\beta<\alpha$.\comm{TRIPLE CHECK} Put $\clB=\clA\restr \beta$. By
the induction hypothesis, there exists a finite tuned refinement $\clC$ of $\clB$.\comm{tuned where?}
Then the partition of $\alpha$ induced by $\clA\cup\clC$ is the required refinement of $\clA$.
\qed
\end{proof}

\begin{question}
Let $(\alpha_i)_{i\in I}$ be a finite family of ordinals. Are the algebras of the direct products $\prod_{i\in I}(\alpha_i,\leq)$,
$\prod_{i\in I}(\alpha_i,<)$ locally finite?
Do the logics of these products have the finite model property?
\end{question}

\bibliographystyle{alpha}
\bibliography{omega-n}


\end{document}